\documentclass[11pt, a4paper]{amsart}

\usepackage{amscd,verbatim}
\usepackage{amssymb}

\usepackage{bm}

\usepackage{amssymb, amsmath}
\usepackage[all]{xy}
\usepackage[inline]{enumitem}
\usepackage{hyphenat}
\usepackage[colorlinks,linkcolor=blue,citecolor=blue,urlcolor=red]{hyperref}
\usepackage[dvipsnames]{xcolor}
\usepackage{mathtools}
\usepackage{tikz-cd}
\usepackage[labelformat=empty]{caption}
\usepackage{xfrac}
 \usepackage[ left=3cm, right=3cm, top = 2.8cm, bottom = 2.8cm ]{geometry}
\usepackage{enumitem}

\makeatletter
\def\@tocline#1#2#3#4#5#6#7{\relax
  \ifnum #1>\c@tocdepth 
  \else
    \par \addpenalty\@secpenalty\addvspace{#2}%
    \begingroup \hyphenpenalty\@M
    \@ifempty{#4}{%
      \@tempdima\csname r@tocindent\number#1\endcsname\relax
    }{%
      \@tempdima#4\relax
    }%
    \parindent\z@ \leftskip#3\relax \advance\leftskip\@tempdima\relax
    \rightskip\@pnumwidth plus4em \parfillskip-\@pnumwidth
    #5\leavevmode\hskip-\@tempdima
      \ifcase #1
      \or\or \hskip 2em \or \hskip 2em \else \hskip 3em \fi%
      #6\nobreak\relax
    \dotfill\hbox to\@pnumwidth{\@tocpagenum{#7}}\par
    \nobreak
    \endgroup
  \fi}
\makeatother

\newcommand{\A}{\mathbf{A}}

\newcommand{\G}{\mathbf{G}}

\newcommand{\N}{\mathbb{N}}
\renewcommand{\P}{\mathbf{P}}

\newcommand{\Z}{\mathbb{Z}}

\newcommand{\sK}{\mathcal{K}}

\newcommand{\sP}{\mathcal{P}}

\newcommand{\Cor}{\operatorname{\mathbf{Cor}}}

\newcommand{\HI}{\operatorname{\mathbf{HI}}}

\newcommand{\Ext}{\operatorname{Ext}}
\newcommand{\ul}[1]{{\underline{#1}}}

\newcommand{\NST}{\operatorname{\mathbf{NST}}}

\newcommand{\DM}{\operatorname{\mathbf{DM}}}

\newcommand{\Map}{\operatorname{map}}
\newcommand{\Hom}{\operatorname{Hom}}
\newcommand{\uHom}{\operatorname{\underline{Hom}}}

\newcommand{\Pic}{\operatorname{Pic}}

\newcommand{\Spec}{\operatorname{Spec}}

\newcommand{\Comp}{\operatorname{Comp}}
\newcommand{\Sm}{\operatorname{\mathbf{Sm}}}

\newcommand{\Shv}{\operatorname{\mathbf{Shv}}}

\newcommand{\pro}[1]{\text{\rm pro}_{#1}\text{\rm--}}
\newcommand{\tr}{{\operatorname{tr}}}

\newcommand{\dlog}{{\operatorname{dlog}}}

\newcommand{\fin}{{\operatorname{fin}}}

\newcommand{\Zar}{{\operatorname{Zar}}}
\newcommand{\Nis}{{\operatorname{Nis}}}
\newcommand{\et}{{\operatorname{\acute{e}t}}}

\newcommand{\id}{{\operatorname{Id}}}

\renewcommand{\lim}{\operatornamewithlimits{\varprojlim}}
\newcommand{\colim}{\operatornamewithlimits{\varinjlim}}

\newcommand{\ol}{\overline}

\renewcommand{\phi}{\varphi}
\renewcommand{\epsilon}{\varepsilon}

\newcommand{\MNS}{\operatorname{\mathbf{MNS}}}
\newcommand{\MNST}{\operatorname{\mathbf{MNST}}}

\newcommand{\MSm}{\operatorname{\mathbf{MSm}}}

\newcommand{\MPST}{\operatorname{\mathbf{MPST}}}
\newcommand{\CI}{\operatorname{\mathbf{CI}}}

\newcommand{\Sq}{{\operatorname{\mathbf{Sq}}}}
\newcommand{\MSmsq}{{(\MSm)^{\Sq}}}

\newcommand{\bcube}{{\ol{\square}}}
\newcommand{\cube}{\square}

\newcommand{\M}{\mathbf{M}}

\newcommand{\ulMSm}{\operatorname{\mathbf{\underline{M}Sm}}}
\newcommand{\ulMNS}{\operatorname{\mathbf{\underline{M}NS}}}
\newcommand{\ulMPS}{\operatorname{\mathbf{\underline{M}PS}}}

\newcommand{\ulMNST}{\operatorname{\mathbf{\underline{M}NST}}}
\newcommand{\ulMCor}{\operatorname{\mathbf{\underline{M}Cor}}}

\newcounter{spec}
{\end{list}}%

\newtheorem{lemma}{Lemma}[section]
\newtheorem{thm}[lemma]{Theorem}

\newtheorem{prop}[lemma]{Proposition}

\newtheorem{cor}[lemma]{Corollary}

\theoremstyle{definition}

\newtheorem{constr}[lemma]{Construction}

\theoremstyle{remark}

\newtheorem{remark}[lemma]{Remark}

\numberwithin{equation}{section}

\numberwithin{equation}{lemma}

\colorlet{LightRubineRed}{RubineRed!70!}

\def\lSm{\mathbf{lSm}}
\def\SmlSm{\mathbf{SmlSm}}
\def\Sm{\mathbf{Sm}}

\newcounter{elno}

\begin{document}

\def\aNis{a_{\Nis}}
\def\ulaNis{\underline{a}_{\Nis}}
\def\ulasNis{\underline{a}_{s,\Nis}}
\def\ulaNisfin{\underline{a}^{\fin}_{\Nis}}
\def\ulasNisfin{\underline{a}^{\fin}_{s,\Nis}}
\def\asNis{a_{s,\Nis}}
\def\ulasNis{\underline{a}_{s,\Nis}}
\def\qaq{\quad\text{ and }\quad}
\def\limcat#1{``\underset{#1}{\lim}"}
\def\Comp{\Comp^{\fin}}
\def\ulc{\ul{c}}
\def\ulb{\ul{b}}
\def\ulgam{\ul{\gamma}}
\def\MSm{\operatorname{\mathbf{MSm}}}
\def\MsigmaS{\operatorname{\mathbf{MsigmaS}}}
\def\ulMSm{\operatorname{\mathbf{\ul{M}Sm}}}
\def\ulMsigmaS{\operatorname{\mathbf{\ul{M}NS}}}

\def\ulMPS{\operatorname{\mathbf{\ul{M}PS}}}

\def\ulMsigmaS{\operatorname{\mathbf{\ul{M}PS}_\sigma}}
\def\ulMsigmaSTfin{\operatorname{\mathbf{\ul{M}PST}^{\fin}_\sigma}}
\def\ulMsigmaST{\operatorname{\mathbf{\ul{M}PST}_\sigma}}
\def\MsigmaS{\operatorname{\mathbf{MPS}_\sigma}}
\def\MsigmaST{\operatorname{\mathbf{MPST}_\sigma}}
\def\MsigmaSTfin{\operatorname{\mathbf{MPST}^{\fin}_\sigma}}

\def\ulMNS{\operatorname{\mathbf{\ul{M}NS}}}
\def\ulMNSTfin{\operatorname{\mathbf{\ul{M}NST}^{\fin}}}
\def\ulMNSfin{\operatorname{\mathbf{\ul{M}NS}^{\fin}}}
\def\ulMNST{\operatorname{\mathbf{\ul{M}NST}}}
\def\ulMEST{\operatorname{\mathbf{\ul{M}EST}}}
\def\MNS{\operatorname{\mathbf{MNS}}}
\def\MNST{\operatorname{\mathbf{MNST}}}
\def\MEST{\operatorname{\mathbf{MEST}}}
\def\MNSTfin{\operatorname{\mathbf{MNST}^{\fin}}}
\def\RSC{\operatorname{\mathbf{RSC}}}
\def\NST{\operatorname{\mathbf{NST}}}
\def\EST{\operatorname{\mathbf{EST}}}

\newcommand{\NS}{{\operatorname{\mathrm{NS}}}}

\def\LogRec{\operatorname{\mathbf{LogRec}}}
\def\Ch{\operatorname{\mathrm{Ch}}}

\def\MSmsq{\MSm^{\Sq}}
\def\Comp{\operatorname{\mathbf{Comp}}}
\def\uli{\ul{i}}
\def\ulis{\ul{i}_s}
\def\is{i_s}
\def\qfor{\text{ for }\;\;}
\def\CIlog{\operatorname{\mathbf{CI}}^{\mathrm{log}}}
\def\CIltr{\operatorname{\mathbf{CI}}^{\mathrm{ltr}}}
\def\CIt{\operatorname{\mathbf{CI}}^\tau}
\def\CItsp{\operatorname{\mathbf{CI}}^{\tau,sp}}
\def\ltr{\mathrm{ltr}}

\def\kX{\mathfrak{X}}
\def\kY{\mathfrak{Y}}
\def\kC{\mathfrak{C}}

\def\otCIsp{\otimes_{\CI}^{sp}}
\def\otCINissp{\otimes_{\CI}^{\Nis,sp}}

\def\hM#1{h_0^{\bcube}(#1)}
\def\hMNis#1{h_0^{\bcube}(#1)_{\Nis}}
\def\hMM#1{h^0_{\bcube}(#1)}
\def\hMw#1{h_0(#1)}
\def\hMwNis#1{h_0(#1)_{\Nis}}

\def\hetrec{h_{0, \et}^{\mathbf{rec}}}
\def\hetcube{h_{0, \et}^{\cube}}

\def\ihF#1{F^{#1}}
\def\ihFA{\ihF {\A^1}}

\def\istm{\iota_{st,m}}
\def\im{\iota_m}
\def\est{\epsilon_{st}}
\def\tL{\tilde{L}}
\def\tX{\tilde{X}}
\def\tY{\tilde{Y}}
\def\omegaCI{\omega^{\CI}}
\def\qwith{\;\text{ with} }
\def\aVNis{a^V_\Nis}
\def\ulMCorls{\ulMCor_{ls}}

\def\Zinf{Z_\infty}
\def\Einf{E_\infty}
\def\Xinf{X_\infty}
\def\Yinf{Y_\infty}
\def\Pinf{P_\infty}

\def\Lot{{\cubegm\otimes\cubegm}}

\def\Ln#1{\Lambda_n^{#1}}
\def\tLn#1{\widetilde{\Lambda_n^{#1}}}
\def\tild#1{\widetilde{#1}}
\def\otuCINis{\otimes_{\underline{\CI}_\Nis}}
\def\otCI{\otimes_{\CI}}
\def\otCINis{\otimes_{\CI}^{\Nis}}
\def\tF{\widetilde{F}}
\def\tG{\widetilde{G}}
\def\bcubered{\bcube^{\mathrm{red}}}
\def\cubegm{\bcube^{(1)}}
\def\cubegma{\bcube^{(a)}}
\def\cubegmb{\bcube^{(b)}}
\def\cubegmred{\bcube^{(1)}_{red}}
\def\cubegmreda{\bcube^{(a)}_{red}}
\def\cubegmredb{\bcube^{(b)}_{red}}

\def\LT{\bcube^{(1)}_{T}}
\def\LU{\bcube^{(1)}_{U}}
\def\LV{\bcube^{(1)}_{V}}
\def\LW{\bcube^{(1)}_{W}}
\def\LTred{\bcube^{(1)}_{T,red}}
\def\Lred{\bcube^{(1)}_{red}}
\def\LTred{\bcube^{(1)}_{T,red}}
\def\LUred{\bcube^{(1)}_{U,red}}
\def\LVred{\bcube^{(1)}_{V,red}}
\def\LWred{\bcube^{(1)}_{W,red}}
\def\PP{\P}
\def\AA{\A}

\def\LL{\bcube^{(2)}}
\def\LLred#1{\bcube^{(2)}_{#1,red}}
\def\LLredd{\bcube^{(2)}_{red}}
\def\Lredd#1{\bcube_{#1,red}}

\def\Lnredd#1{\bcube^{(#1)}_{red}}

\def\LLT{\bcube^{(2)}_T}
\def\LLTred{\bcube^{(2)}_{T,red}}

\def\LLU{\bcube^{(2)}_U}
\def\LLUred{\bcube^{(2)}_{U,red}}

\def\LLS{\bcube^{(2)}_S}
\def\LLSred{\bcube^{(2)}_{S,red}}
\def\tMCor{\Hom_{\MPST}}
\def\otHINis{\otimes_{\HI}^{\Nis}}

\def\Sh{\operatorname{\mathbf{Shv}}}
\def\Shv{\operatorname{\mathbf{Shv}}}
\def\PSh{\operatorname{\mathbf{PSh}}}
\def\Shltr{\operatorname{\mathbf{Shv}_{dNis}^{ltr}}}
\def\Shlog{\operatorname{\mathbf{Shv}_{dNis}^{log}}}
\def\Shvlog{\operatorname{\mathbf{Shv}^{log}}}
\def\SmlSm{\operatorname{\mathbf{SmlSm}}}
\def\lSm{\operatorname{\mathbf{lSm}}}
\def\lCor{\operatorname{\mathbf{lCor}}}
\def\SmlCor{\operatorname{\mathbf{SmlCor}}}
\def\PShltr{\operatorname{\mathbf{PSh}^{ltr}}}
\def\PShlog{\operatorname{\mathbf{PSh}^{log}}}
\def\lDM{\operatorname{\mathbf{logDM}^{eff}}}
\def\logDM{\operatorname{\mathbf{log}\mathcal{DM}^{eff}}}
\def\logDMlet{\operatorname{\mathbf{log}\mathcal{DM}^{eff}_{\mathrm{l\acute{e}t}}}}
\def\logDMone{\operatorname{\mathbf{log}\mathcal{DM}^{eff}_{\leq 1}}}
\def\logCI{\mathbf{logCI}^{\ltr}} 

\def\DM{\operatorname{\mathbf{DM}^{eff}}}
\def\DMinf{\operatorname{\mathcal{DM}^{eff}}}
\def\lDA{\operatorname{\mathbf{logDA}^{eff}}}
\def\logDA{\operatorname{\mathbf{log}\mathcal{DA}^{eff}}}
\def\DA{\operatorname{\mathbf{DA}^{eff}}}
\def\Log{\operatorname{\mathcal{L}\textit{og}}}
\def\Rsc{\operatorname{\mathcal{R}\textit{sc}}}
\def\Pro{\mathrm{Pro}\textrm{-}}
\def\pro{\mathrm{pro}\textrm{-}}
\def\dg{\mathrm{dg}}
\def\plim{\mathrm{``lim"}}
\def\ker{\mathrm{ker}}
\def\coker{\mathrm{coker}}

\def\Alb{\operatorname{Alb}}
\def\bAlb{\mathbf{Alb}}
\def\Gal{\operatorname{Gal}}

\def\hofib{\mathrm{hofib}}
\def\triv{\mathrm{triv}}
\def\ABl{\mathcal{A}\textit{Bl}}
\def\divsm#1{{#1_\mathrm{div}^{\mathrm{Sm}}}}

\def\cA{\mathcal{A}}
\def\cB{\mathcal{B}}
\def\cC{\mathcal{C}}
\def\cD{\mathcal{D}}
\def\cE{\mathcal{E}}
\def\cI{\mathcal{I}}
\def\cS{\mathcal{S}}
\def\cM{\mathcal{M}}
\def\cO{\mathcal{O}}
\def\cP{\mathcal{P}}

\def\XP{X \backslash \sP}
\def\M0a{{}^t\cM_0^a}
\newcommand{\Ind}{{\operatorname{Ind}}}

\def\Xkbar{\overline{X}_{\overline{k}}}
\def\dx{{\rm d}x}

\newcommand{\dNis}{{\operatorname{dNis}}}
\newcommand{\loget}{{\operatorname{l\acute{e}t}}}
\newcommand{\ABNis}{{\operatorname{AB-Nis}}}
\newcommand{\sNis}{{\operatorname{sNis}}}
\newcommand{\sZar}{{\operatorname{sZar}}}
\newcommand{\set}{{\operatorname{s\acute{e}t}}}
\newcommand{\cofib}{\mathrm{Cofib}}

\newcommand{\Gmlog}{\G_m^{\log}}
\newcommand{\Gmlogred}{\overline{\G_m^{\log}}}

\newcommand{\varcolim}{\mathop{\mathrm{colim}}}
\newcommand{\varlim}{\mathop{\mathrm{lim}}}
\newcommand{\tensor}{\otimes}

\newcommand{\eq}[2]{\begin{equation}\label{#1}#2 \end{equation}}
\newcommand{\eqalign}[2]{\begin{equation}\label{#1}\begin{aligned}#2 \end{aligned}\end{equation}}

\def\varplim#1{\text{``}\varlim_{#1}\text{''}}
\def\det{\mathrm{d\acute{e}t}}

	\address{Institut für Mathematik, Universität Heidelberg, MATHEMATIKON INF 205, 69120  Heidelberg}
	\email{merici@mathi.uni-heidelberg.de}
	
	\thanks{This project was supported by the RCN project 313472 \emph{EMOHO - Equations in Motivic Homotopy}, the MSCA-PF project 101103309 \emph{MIPAC - Motivic Integral p-adic Cohomologies} and the Deutsche Forschungsgemeinschaft (DFG, German Research Foundation)
TRR 326 \textit{Geometry and Arithmetic of Uniformized Structures}, project number 444845124.}

	\title{Some computations in the heart of the homotopy t-structure on logarithmic motives}
	\author{Alberto Merici}
    \begin{abstract}
    In this note we will illustrate a method for computing the $\pi_0$ of the effective log motive of a smooth and proper variety over a perfect field $k$ and show that it is $\A^1$-invariant. We will apply this to compute the first homotopy groups of $\P^1$ to show that the stripping functor from log motivic sheaves to (usual) Nisnevich sheaves with transfers is fully faithful.      
	\end{abstract}
    
	\maketitle
\section{Introduction}

Since their introduction in \cite{BPO}, log motives have proved to be very useful to compute non-$\A^1$-invariant cohomologies using methods from motivic homotopy theory. The main motivic technique is the log version of Morel--Voevodsky purity of \cite[Theorem 3.2.21]{BPO-SH}, which we now recall. Let $S\in \mathbf{Sch}$ and $X=(\ul{X},\partial X)$ a log scheme log smooth over $S$ with $\ul{X}$ smooth over $S$ and $\partial X$ supported on a simple normal crossing divisor $D_1+\ldots D_n$. Let $Z$
be a smooth closed subscheme of $\ul{X}$ having strict normal crossing with $D_1 +\ldots D_n$. Then there is a cofiber sequence functorial in $Z\subseteq X$\[
M^{\rm eff}(Bl_Z(X),E) \to M^{\rm eff}(X) \to \mathrm{Th}(N_Z X),
\]
where $\mathrm{Th}(N_Z X)$ is the Thom space of the normal bundle of $Z\subseteq \ul{X}$, with an appropriate log structure. 

If $X=(\ul{X},\triv)$ is affine and with trivial log structure and we choose $Z=D_1$ a smooth divisor, the sequence above gives\[
M^{\rm eff}(\ul{X},D_1)\to M^{\rm eff}(\ul{X},\triv)\xrightarrow{\mathrm{Gys}} \dfrac{M^{\rm eff}(Z\times \P^1,\triv)}{M^{\rm eff}(Z,\triv)},
\]
and the terms in the middle and on the right have trivial log structure: this raises the question on whether one can compute recursively $M^{\rm eff}(X)$ using only data coming from classical (non-logarithmic) schemes. The obstacle to this is that the map $\mathrm{Gys}$ above is induced by a zig-zag of maps of log schemes with non-trivial log structures (as in \cite[Theorem 3.2.21]{BPO-SH}). 

Let $k$ be a perfect field. In \cite{BindaMerici} it was shown that the categories of effective logarithmic motives over a field $\logDM(k,\Z)$ and $\logDA(k,\Z)$ admit a homotopy $t$-structure, whose hearts are the category of strictly $\bcube$-invariant Nisnevich sheaves with and without transfers, respectively, i.e. sheaves such that the projection induces an equivalence $H^q_{\dNis}(X,F)\cong H^q_{\dNis}(X\times\bcube,F)$ fo all $q\geq 0$. The lifting problem above in this situation is more treatable, indeed if $\omega_\sharp$ denotes the functor that sends a sheaf $F$ on log schemes to the sheaf on (usual) schemes\[
X\in \Sm_k \mapsto F(X,\triv),
\] 
then this restriction to the hearts of the $t$-structure is faithful and exact, and in \cite[Theorem 3.26]{mericicrys} the precise obstruction to its fullness has been computed (see Theorem \ref{thm:Gysin} for the explicit condition).

In this note we show that in the heart of the homotopy $t$-structure $\logCI(k,\Z)=\logDM(k,\Z)^{\heartsuit}$, the obstruction is an empty condition (see Theorem \ref{thm:almost-conj}):
\begin{thm}\label{thm:main-intro}
		Let $F,G\in \logCI(k,\Z)$, and let $\phi\colon \omega_\sharp F\to \omega_\sharp G$ be a map in $\Sh_{\Nis}^{\tr}(k,\Z)$. Then $\phi$ lifts uniquely to a map $F\to G$ in $\logCI(k,\Z)$. In particular the functor $
        \logCI(k,\Z)\xrightarrow{\omega_\sharp\iota} \Sh_{\Nis}^{\tr}(k,\Z)$
        is fully faithful and exact.
\end{thm}
This result was first announced in in \cite[Section 7]{BindaMerici} assuming resolutions of singularities, but the proof contained a gap pointed out in \cite{BindaMericierratum}: then this was formulated as a conjecture in \cite[Conjecture 0.2]{BindaMericierratum}. 

The main input of this is a method to compute $h_0^{\bcube}(\ul{X},D)$ with $\ul{X}$ smooth and proper over $k$, and we deduce that it is $\A^1$-invariant, without any assumption on resolution of singularities (with this assumption, this was proved in a more general form in \cite[Proposition 8.2.8]{BPO}). We use this to compute $\pi_1 M^{\rm eff}(\P^1)$, which will be crucial to show that in the case with transfers the obstruction is an empty condition, and so deduce \cite[Conjecture 0.2]{BindaMericierratum} from \cite[Theorem 3.26]{mericicrys}. This corrects the gap in \cite[Section 7]{BindaMerici}.

In the $\P^1$-stable situation, we expect this result to be useful in the comparison of log motivic spectra with the $\P^1$-spectra of Annala--Hoyois--Iwasa \cite{AHI}: in this context the Gysin maps are contained in a work in progress by Longke Tang (see \cite{AHI-Atiyah}). 

\subsection*{Acknowledgements}
	This article first appeared as an appendix to \cite{mericicrys}, but due to its length and independent interest we decided to separate the two parts. The author would like to thank the anonymous referees of \cite{mericicrys} and of this paper for their meticulous analysis, providing helpful comments which filled some gaps in the arguments and led to an improved presentation. The main results were obtained while the author was a postdoc supported by the RCN project \emph{EMOHO} at the University of Oslo, and the final version was settled with the support of the MSCA-PF \emph{MIPAC} carried out at the University of Milan, and the CRC/TRR \emph{GAUS}, at Universit\"at Heidelberg. The author is very thankful for the hospitality and the great work environment.
\subsection*{Notation}
For $F\to G$ a map in a stable $\infty$-category, we will denote by $G/F$ or $\dfrac{G}{F}$ its homotopy cofiber. For $\cC$ a stable $\infty$-category equipped with a $t$-structure, we will always consider implicit the embedding $\cC^{\heartsuit}\subseteq \cC$: in order to improve readability we will avoid writing $F[0]$ for an object concentrated in degree $0$ and simply write $F$ if this does not rise confusion, and only write $F[i]$ for $i\neq 0$. In particular, if $F\hookrightarrow G$ is a monomorphism in $\cC^{\heartsuit}$, we will use the notation $G/F$ or $\dfrac{G}{F}$ to indicate the cokernel in $\cC^{\heartsuit}$, which is also equivalent to the cofiber.
\section{Recollections on log motives}\label{ssec:logDM} 
	We recall the construction of logarithmic motives over a perfect field of \cite{BPO} and some properties. The standard reference for log schemes is \cite{ogu}. For a log scheme $X$, we denote by $\ul{X}$ the underlying scheme of $X$, by $\partial X$ the log structure and by $|\partial X|$ its support, seen as a reduced closed subscheme of $X$.
    
	We denote by ${\SmlSm}(k)$ the category of fs log smooth log schemes over the log scheme $(\Spec(k),\triv)$ such that $\ul{X}$ is smooth over $k$. By e.g. \cite[Lemma A.5.10]{BPO}, then in this case $\partial X$ is supported on a strict normal crossing divisor on $\ul{X}$ and  $\partial X$ is isomorphic to the compactifying log structure associated to the open embedding $\ul{X}-|\partial X| \hookrightarrow \ul{X}$. If $D$ is a strict normal crossing divisor on $\ul{X}$, we will often write $(\ul{X},D)\in \SmlSm(k)$ meaning the log scheme with compactifying log structure associated to $\ul{X}-|D|\hookrightarrow \ul{X}$.

    Following \cite{BPO}, we denote by $\lCor(k)$ the category of finite log correspondences over $k$. It is a variant of the Suslin--Voevodsky category of finite correspondences $\Cor(k)$. It has the same objects as $\SmlSm(k)$\footnote{Notice that this notation conflicts with the notation of \cite{BPO} where the objects were the same as $\lSm(k)$, although the categories of sheaves are the same in light of \cite[Lemma 4.7.2]{BPO}}, and morphisms are given by the free abelian group generated by elementary correspondences  $V^o\subset (X- \partial X) \times (Y- \partial Y)$ such that the closure $V\subset \ul{X}\times \ul{Y}$ is finite and surjective over (a component of) $\ul{X}$ and such that there exists a morphism of log schemes $V^N \to Y$, where $V^N$ is the fs log scheme whose underlying scheme is the normalization of $V$ and whose log structure is given by the inverse image log structure along the composition $\underline{V^N} \to \ul{X}\times \ul{Y} \to \ul{X}$. See \cite[2.1]{BPO} for more details, and for the proof that this definition gives indeed a category. 

    This category is equipped with a Grothendieck topology, called the \emph{dividing Nisnevich topology}, generated as a cd-topology by the following: 
    \begin{itemize}
    \item \emph{strict Nisnevich squares}, i.e. pullbacks $
    \begin{tikzcd}
    U_1\times_X U_2\ar[r]\ar[d]&U_1\ar[d]\\
    U_2\ar[r]&X
    \end{tikzcd}$
    such that $\{\ul{U_i}\to \ul{X}\}$ is a Nisnevich square and $\partial U_i = \partial X_{|U_i}$ 
    \item \emph{dividing squares}, i.e. pullbacks $
    \begin{tikzcd}
    \emptyset\ar[r]\ar[d]&X\ar[d]\\
    \emptyset\ar[r]&Y
    \end{tikzcd}$
    where $f$ is a surjective proper log \'etale monomorphism.
    \end{itemize}
	See \cite[Remark 2.2.4]{BPO-SH} for more details on this. 
    
	We denote by $\PShltr(k, \Z)$ the category of abelian presheaves on $\lCor(k)$: such presheaves are called \emph{presheaves with log transfers}. We let $\mathbf{Shv}_{\dNis}^{\rm ltr}(k, \Z)\subseteq \PShltr(k, \Z)$ the category of dividing Nisnevich sheaves with log transfers, and for $U\in \SmlSm(k)$ we let $\Z_{\ltr}(U)$ be the representable sheaf. Analogously to the case of Voevodsky's finite correspondences, it has the structure of a Grothendieck abelian category \cite[Theorem 4.5.7 and \S 4.2]{BPO}, so we let $\cD(\mathbf{Shv}_{\dNis}^{\rm ltr}(k, \Z))$, be its derived category.

    Finally (see \cite[\S 5.2]{BPO}), let $\bcube:=(\P^1,\infty)$. We denote $\mathbf{log}\mathcal{DM}^{\textrm{eff}}(k,\Z)$ the localization of $\cD(\mathbf{Shv}_{\dNis}^{\rm ltr}(k, \Z))$) with respect to the class of maps
		\[ \Z_{\ltr}(\bcube\times X)[n]\to \Z_{\ltr}(X)[n]
		\]
	for all $X\in \SmlSm(k)$ and $n\in \Z$. 
	We let $L_{(\dNis,\bcube)}$
	be the localization functor and for $X\in \SmlSm(k)$, we will let $M^{\rm eff}(X) = L_{(\dNis,\bcube)}(\Z_{\ltr}(X))$.

	We recall the following result \cite[Theorem 5.7]{BindaMerici}:
	\begin{thm}\label{thm:t-structure}
		The standard $t$-structure of $\cD(\Shv^{\rm ltr}_{\dNis}(k, \Z))$ induces an accessible $t$-structure on $\logDM(k,\Z)$ compatible with filtered colimits in the sense of \cite[Definition 1.3.5.20]{HA}, called the \emph{homotopy $t$-structure}.
	\end{thm}
	
	We denote by $\logCI(k,\Z)$ its heart, which is then identified with the full subcategory of $\Shv^{\rm ltr}_{\dNis}(k, \Z)$ of sheaves $F$ such that for all $X\in \SmlSm$, the projection induces an isomorphism $R\Gamma(X,F)\simeq R\Gamma(X\times \bcube,F)$ and it is a Grothendieck abelian category. The inclusion
	\[
	\iota \colon \logCI(k,\Z)\hookrightarrow \Shv^{\rm ltr}_{\dNis}(k, \Z)
	\]
	admits a left adjoint $h_0^{\bcube}$, given by the formula $F\mapsto \pi_0(L_{(\dNis,\bcube)}(F[0]))$, and a right adjoint $h^0_{\bcube}$  (see \cite[Proposition 5.8]{BindaMerici}), in particular $\iota$ preserves all limits and colimits. In general, we let $h_i^{\bcube}(F):= \pi_i(L_{(\dNis,\bcube)}(F[0]))\in \logCI(k,\Z)$.
    
    We have the following purity result \cite[Theorem 5.10]{BindaMerici}:
    \begin{thm}\label{thm:purity}
        Let $X\in \SmlSm(k)$ and let $\ul{U}\subseteq \ul{X}$, $\partial U := \partial X_{\ul{U}}$ and $U:=(\ul{U},\partial U)$. Then for all $F\in \logCI(k,\Z)$ the restriction $F(X)\to F(U)$ is injective.
    \end{thm}
    Recall from  \cite[(4.3.4)]{BPO} that the functor $\omega:X\mapsto X^{\circ}$ induces an adjunction
	\begin{equation}\label{omegaadjunction}
		\begin{tikzcd}
			\Shv_{\dNis}^{\ltr}(k,\Z)\arrow[rr,shift left=1.5ex,"\omega_\sharp"]&& \Shv_{\Nis}^{\tr}(k,\Z)\arrow[ll,"\omega^*"],
		\end{tikzcd}
	\end{equation} 
	where the right hand side is Voevodsky's category of Nisnevich sheaves with transfers. These functors are very explicit: for $Y\in \SmlSm(k)$, $\omega_{\sharp} F(Y) = F(Y,\textrm{triv})$ and for $X\in \Sm(k)$, $\omega^*F(X)=F(\underline{X}-|\partial X|)$. Since $\omega$ is monoidal by construction, $\omega_\sharp$ is monoidal, and it has a left adjoint $\omega^\sharp$ given by the formula $\omega^\sharp \Z_{\tr}(X) = \Z_{\ltr}(X,\triv)$ and extended by colimits. By construction, the functors $\omega^\sharp$ and $\omega_\sharp$ are exact so we will also call denote their derived functor $\omega^\sharp$ and $\omega_\sharp$ without confusion. 
    
    Let $F\in \logCI(k,\Z)$: by purity, Theorem \ref{thm:purity}, the unit of the adjunction gives an injective map\begin{equation}\label{eq:purity-omega}
        \iota F\hookrightarrow \omega^*\omega_\sharp \iota F
    \end{equation}
    By construction, the map above is an isomorphism if and only if $\omega_\sharp \iota F$ is (strictly) $\A^1$-invariant in the sense of Voevodsky. In particular, let $\HI(k,\Z)$ be the category of (strictly) $\A^1$-invariant Nisnevich sheaves with transfers: we have that if $F\in \HI(k,\Z)$ if and only if $\omega^*F\in \logCI(k,\Z)$.
    
    Combining all these properties, we obtain that (see \cite[Proposition 0.1]{BindaMericierratum}:
    \begin{thm}\label{thm:conserva}
        The functor $\omega_\sharp \iota\colon \logCI(k,\Z)\to \Sh^{\tr}(k,\Z)$ is faithful and exact, in particular it is conservative. 
    \end{thm}
    We will use several times the following immediate corollary of the theorem above:
    \begin{cor}\label{cor:dio-cane}
        Let $f\colon F\to G$ be a map in $\logCI(k,\Z)$. Then $\omega_\sharp \iota f$ is injective (resp. surjective) if and only if $f$ is injective (resp. surjective)
        \proof Let $Q$ be the kernel (resp. cokernel) of $f$ in $\logCI(k,\Z)$, then $f$ is injective (resp. surjective) if and only if $Q = 0$. Since $\omega_\sharp\iota $ is exact, $\omega_\sharp \iota Q$ is the kernel (resp. cokernel) of $\omega_\sharp \iota f$. Since $\omega_\sharp\iota$ is conservative, we conclude that $\omega_\sharp Q = 0$ if and only if $Q=0$, which concludes the proof.
    \end{cor}
    
    The goal of this paper is to improve Theorem \ref{thm:conserva} and prove that $\omega_\sharp \iota$ is indeed fully faithful.
    To do so, we will use some techniques from \cite[\S 3]{mericicrys}, that we now recall.
    The first useful result is the following (see \cite[Lemma 3.1]{mericicrys}):
    \begin{lemma}\label{lem:vanish-coh}
        Let $F\in \logCI(k,\Z)$, $K$ a function field over $k$ and $U\subseteq \A^1_K$ open. Then $H^q(U,F) =0$ for $q>0$. 
    \end{lemma}
As recalled in the introduction, $\logDM(k,\Z)$ has Gysin triangles. For $K$ a function field over $k$, $\ul{U}\subseteq \P^1_{K}$ affine open containing $0$, the following commutative square in $\cD(\Sh_{\dNis}^{\ltr}(k,\Z))$
\begin{equation}\label{eq:gysin-pushout}
\begin{tikzcd}
    \Z_{\ltr}(U,0)\ar[r]\ar[d,"j"]& \Z_{\ltr}(U,\triv)\ar[d,"\ul{j}"]\\
    \Z_{\ltr}(\P^1_{K},0)\ar[r]&\Z_{\ltr}(\P^1_{K},\triv),
\end{tikzcd}
\end{equation}
becomes cartesian in $\logDM(k,\Z)$ after applying $L_{(\dNis,\bcube)}$ by \cite[Proposition 2.4.8]{BPO-SH}, hence the Gysin map for $0\subseteq U$ is equivalent (after applying $L_{(\dNis,\bcube)}$) to the composition
\begin{equation}\label{eq:gys-A1}
\Z_{\ltr}(U,\triv)\to \dfrac{\Z_{\ltr}(U,\triv)}{\Z_{\ltr}(U,0)}\to \dfrac{\Z_{\ltr}(\P^1_{K},\triv)}{\Z_{\ltr}(\P^1_{K},0)}.
\end{equation}
    From the construction of \cite[(3.16)]{mericicrys}, the Gysin maps above equip $F\in \logCI(k,\Z)$ with residue maps $F(\eta_x) \to F(\A^1_{k(x)})/F(k(x))$ for every DVR essentially smooth over $k$ with fraction field $\eta$ and residue field $k(x)$ (see the proof of Proposition \ref{prop;star} where the construction is recalled and used), then the obstruction is encoded by these residue maps (see \cite[Proposition 3.19]{mericicrys}):
    \begin{prop}\label{prop:residues}
        For $F, G \in \logCI(k,\Z)$, any map $\phi \colon \omega_\sharp F \to \omega_\sharp G$ that respects the residues lifts uniquely to a map of log sheaves $\phi \colon F \to G$.
    \end{prop}
    Let $F\in \logCI(k,\Z)$, so by $\bcube$-invariance we have that \[
	\Map(\Z_{\ltr}(\P^1_K,\infty)[-1],F)\simeq\Map(\Z_{\ltr}(\P^1_K,0)[-1],F)\simeq F(K)[1].
    \] 
    Recall that by \cite[Proposition 2.4.9]{BPO-SH}, the choice of the $K$-ratioanl point $[1:1]$ gives an isomorphism $F(\P^1_K,\triv)\cong F(K,\triv)$ and a splitting \[
    \Map(\Z_{\ltr}(\P^1_{K},\triv)[-1],F)\simeq F(K,\triv)[1]\oplus \tau_{\leq -1}\Map(\Z_{\ltr}(\P^1_{K},\triv)[-1],F)
    \] 
    By cohomological dimension, we have $\tau_{\leq -1}\Map(\Z_{\ltr}(\P^1_{K},\triv)[-1],F)\simeq H^1_{\Nis}(\P^1_{K},\omega_{\sharp} F)$,
    therefore \[
    \Map\Bigl(\dfrac{\Z_{\ltr}(\P^1_{K},\triv)}{ \Z_{\ltr}(\P^1_K,0)}[-1],F\Bigr)\simeq H^1_{\Nis}(\P^1_{K},\omega_{\sharp} F).
    \]
    Moreover, by \cite[Lemma 3.1]{mericicrys}, we have that $H^q_{\dNis}(U,F)=0$ for $q>0$ and $U\subseteq \A^1_{K}$ dense open, therefore\[
    \Map(\Z_{\ltr}(U,\triv)[-1],F)\simeq \omega_\sharp F(U)[1].
    \] 
    Therefore, by applying the functor $\Map((-)[-1],F)$ the map \eqref{eq:gys-A1} above gives for all $U\subseteq \A^1_{K}$ open containing $0$ a map
    \begin{equation}\label{eq:Thom}
			H^1_{\Nis}(\P^1_{K},\omega_{\sharp} F)\simeq \Map\Bigl(\dfrac{\Z_{\ltr}(\P^1_{K},\triv)}{ \Z_{\ltr}(\P^1_K,0)}[-1],F\Bigr)\to \Map(\Z_{\ltr}(U,\triv)[-1],F)\simeq \omega_\sharp F(U)[1]
	\end{equation}
    compatible with open embeddings $V\subseteq U$.
	Let $i\colon \{0\}\hookrightarrow\A^1_{K}$: then by restricting to the small Zariski site of $\A^1_K$, the map \eqref{eq:Thom} gives a map in $\cD(\Sh((\A^1_{K})_{\Zar},\Z)$\[
	\mathrm{Gys}_{0}^F\colon i_*H^1({\P^1_{K}},\omega_\sharp F)\to F_{\A^1_{K}}[1],
	\] 
	where $H^1(\P^1_{K},\omega_{\sharp}F)$ on the left hand side is seen as a (constant) sheaf on the small site of $\{0\}$.
    We have the following theorem (see \cite[Theorem 3.26]{mericicrys}):
	\begin{thm}\label{thm:Gysin}
		Let $F,G\in \logCI(k,\Z)$, and let $\phi\colon \omega_\sharp F\to \omega_\sharp G$ be a map of sheaves with transfers. If for all function fields $K$ over $k$ the maps $\mathrm{Gys}_0^F$ and $\mathrm{Gys}_0^G$ fit into a commutative diagram in $\cD(\Sh((\A^1_{K})_{\Zar}))$\[
		\begin{tikzcd}
			i_*H^1({\P^1_{K}},\omega_\sharp F)\ar[r,"\mathrm{Gys}_{0}^F"]\ar[d,"i_*H^1(\phi_{\P^1_{K}})"'] &F_{\A^1_{K}}[1]\ar[d,"\phi_{\A^1_{K}}"]\\
			i_*H^1({\P^1_{K}},\omega_\sharp G)\ar[r,"\mathrm{Gys}_{0}^G"] &G_{\A^1_{K}}[1],
		\end{tikzcd}
		\]
		then $\phi$ respects the residues, and so there is a (unique) map $\widetilde{\phi}\colon F\to G$ that lifts $\phi$.
        \end{thm}
        The proof of Theorem \ref{thm:main-intro} will reduce to show that the map $\mathrm{Gys}_0^F$ agrees with the map induced by the composition in $\cD(\Sh_{\Nis}^{\tr}(k,\Z))$\[
        \Z_{\tr}(U)[-1]\xrightarrow{\ul{j}} \Z_{\tr}(\P^1)[-1]\xrightarrow{\tau} \G_m,\] 
        where the first map is the open immersion and the second map corresponds to the map given by the tautological bundle in $\Pic(\P^1)$ (see Lemma \ref{lem:gysin-is-downstairs}): this is indeed a map in $\cD(\Sh^{\tr}_{\Nis}(k,\Z))$, so it is functorial for maps $\omega_\sharp F\to \omega_\sharp G$ in $\Sh_{\Nis}^{\tr}(k,\Z)$.
    \def\h0{h_0^{\bcube}}    


\section{A computation of motivic localization}\label{appendix-A}

    In this section, we will show a general method to compute $h_0^{\bcube}(-)$ and we will apply it to compute $h_i^{\bcube}(\P^1,\triv)$ for $i=0,1$, without any assumption on the perfect field $k$.
	
	\begin{constr}\label{constr}
		Let $S\in \mathbf{Shv}^{\ltr}_{\dNis}(k,\Z)$ and let $T\in \logCI(k,\Z)$ with $f\colon S\to T$ such that $\omega_\sharp(f)$ is surjective. 
		Then we have a commutative diagram\[
		\begin{tikzcd}
			\omega_\sharp S\ar[rr]\ar[d] &&\omega_\sharp T\ar[d,equal]\\
			\omega_\sharp \h0(S)\ar[rr,"\omega_\sharp \h0(f)"]&&\omega_\sharp \h0(T)
		\end{tikzcd}
		\]
		In particular $\omega_{\sharp}\h0(f)$ is surjective, so by Corollary \ref{cor:dio-cane} the map $\h0(f)$ is surjective. 
		
		Let $L\colon S\to \h0(S)$ be the localization map, given by the unit of the adjunction $\h0\dashv \iota$, and let $E$ (resp. $E'$) be the kernel of $f$ (resp. $\h0(f)$). We fix the notation in the following commutative diagram, where the rows are by definition exact sequences:
		\begin{equation}\label{eq;diagram}
			\begin{tikzcd}
				0\ar[r] &E\ar[r,"q"]\ar[d,"L_0"] &S\ar[r,"f"]\ar[d,"L"] &T\ar[d,equal]\\
				0\ar[r] &E'\ar[r,"q'"] &\h0(S)\ar[r,"\h0(f)"] &T\ar[r]&0
			\end{tikzcd}
		\end{equation}
		Let $K$ be a function field and for each $t\in K$ consider $i_t\colon (K,\triv)\to \bcube_K$ given by  the inclusion of $[1:t]$ in $\P^1_K$. Notice that for all $F\in \logCI(k,\Z)$ and $t\in K$, the map $i_t^*\colon F(\bcube_K)\to F(K,\triv)$ is an isomorphism inverse of $F(K,\triv)\to F(\bcube_K)$ given by the structural map, so for $t,t'\in K$, $i_t^*=i_{t'}^*$. We have the following commutative diagram:
		\begin{equation}\label{eq;i0-i1}
			\begin{tikzcd}
				&\mathrm{Ker}(L_K)\ar[d]\\
				S(\bcube_K)\ar[r,"i_0^*-i_1^*"]\ar[d]\ar[ur,dotted]&S(K,\triv)\ar[d,"L_K"]\\
				\h0 S(\bcube_K)\ar[r,"i_0^*-i_1^* = 0"]&\h0S(K,\triv)
			\end{tikzcd}
		\end{equation}
	\end{constr}
	Let us consider the following condition:
	\begin{equation}\label{star}
		\tag{$\star$} q_K(E(K,\triv))\subseteq {\rm Im}(i_0^*-i_1^*)
	\end{equation}
	\begin{prop}\label{prop;star}
		In the situation of Construction \ref{constr}, if \eqref{star} is satisfied for every function field $K$, then $\h0(S)\cong T$.
	\end{prop} \begin{proof}
		We need to show that for all $g\colon S\to C$ with $C\in \logCI(k,\Z)$ there is a unique map $\ol{g}\colon T\to C$ that factors $g$ via $f$. Notice that if $\ol{g},\ol{g}'$ such that $\ol{g}\circ f=\ol{g}'\circ f$, by the universal property of $\h0(S)$ we have $\ol{g}\circ \h0(f)=\ol{g}'\circ \h0(f)$, and since $\h0(f)$ is surjective we get $\ol{g}=\ol{g}'$. We are then left to show that such $\ol{g}$ exists: to do so, we will show that there exists a map $T\to \h0(S)$ that factors $S\to \h0(S)$.
		The condition \eqref{star} implies that for all function fields $K$\[
		q'_K\circ (L_0)_K(E(K,\triv))\overset{\eqref{eq;diagram}}{=} L_K\circ q_KE(K,\triv) \overset{\eqref{star}}{\subseteq}L_K((i_0^*-i_1^*)(S(\bcube_K))) \overset{\eqref{eq;i0-i1}}{=}0 .
		\]
		Since $q'_K$ is injective by definition (see \eqref{eq;diagram}), we have that the map $(L_0)_K$ is the zero map. 
        
        By purity, Theorem \ref{thm:purity}, for all $X\in \SmlSm(k)$ with fraction field $K$, $\h0(S)(X)\hookrightarrow \h0(S)(K,\triv)$, so $E'(X)\hookrightarrow E'(K,\triv)$, which implies that $L_0$ is the zero map. So, we have that $E=\ker(L)$, hence the map $L$ factors through  a map $p\colon S/E\to \h0(S)$. 
		By the universal property of $L$, for all $C\in \logCI(k,\Z)$ and $g\colon S\to C$, we have a unique map $g'$ such that $g=g'\circ L$, so the following diagram commutes:\[
		\begin{tikzcd}
			S\ar[rrr,"g"]\ar[dr]&&&C\\
			&S/E\ar[r,"p"]&\h0(S)\ar[ur,swap,"g'"]
		\end{tikzcd}
		\]
		Moreover, the map $u\colon (S/E)\to T$ induces an isomorphism $\omega_{\sharp}(S/E)\cong \omega_{\sharp} T$ since $\omega_\sharp f$ is surjective by assumption and $\omega_\sharp S/E = \omega_\sharp S/\ker(\omega_\sharp f)$, so we get a map 
        \begin{equation}\label{eq:phi}
            \phi = \omega_\sharp p\circ (\omega_\sharp u)^{-1}\colon \omega_\sharp T\to \omega_\sharp \h0(S).
        \end{equation}

        Let $\ul{X}$ be the spectrum of an essentially smooth DVR over $k$ with generic point $\eta$ and closed point $x$ with residue field $k(x)$.
        Since $(S/E)_{\A^1_{k(x)}} \cong T_{\A^1_{k(x)}}$, we have by  Lemma \ref{lem:vanish-coh} that\[
        R\Gamma(\A^1_{k(x)},S/E) \simeq R\Gamma(\A^1_{k(x)},T) \simeq S/E(\A^1_{k(x)})[0].
        \]
        Recall that by Gabber's presentation lemma \cite[Theorem 3.1.1]{CTHK} (see also \cite[Proof of Lemma 4.1]{BindaMerici} for its use in the log context) there is a formally \'etale map $e\colon \ul{X}\to \P^1_{k(x)}$ inducing a strict (pro-)Nisnevich square:	\begin{equation}\label{eq:Gabber}
				\begin{tikzcd}
					(\eta_x,\triv) \ar[r]\ar[d,"{(e^\circ)}"] &(\ul{X},\triv) \ar[d,"\ul{e}"]
                    \\
					(\A^1_{k(x)},\triv)\ar[r] &(\P^1_{k(x)},\triv)
				\end{tikzcd}
			\end{equation} 
            giving for all $F\in \Sh^{\ltr}_{\dNis}(k,\Z)$ a map (we omit the ``$\triv$" for simplicity)\[
            F(\eta_x)\to \dfrac{F(\eta_x)}{R\Gamma(\ul{X},F)}\cong \dfrac{R\Gamma(\A^1_{k(x)},F)}{R\Gamma(\P^1_{k(x)},F)}
            \]
            functorial in $F$, therefore the maps $u$ and $p$ defined above give a commutative diagram\[
		\begin{tikzcd}
			T(\eta_x)\ar[d] &S/E(\eta_x)\ar[d]\ar[r,"p_{\eta_x}"]\ar[l,"\simeq","u_{\eta_x}"'] &\h0(S)(\eta_x)\ar[d]\\
			\dfrac{T(\A^1_{k(x)})}{R\Gamma(\P^1_{k(x)},T)} &\dfrac{S/E(\A^1_{k(x)})}{R\Gamma(\P^1_{k(x)},S/E)}\ar[r]\ar[l,"\simeq"] &\dfrac{\h0(S)(\A^1_{k(x)})}{R\Gamma(\P^1_{k(x)},\h0(S))}.
		\end{tikzcd}
		\]
    By composing with the maps $\Spec(k)\xrightarrow{i_0}\bcube\to (\P^1,\triv)$, the maps $u$ and $p$ give a commutative diagram\[
    \begin{tikzcd}
    \dfrac{T(\A^1_{k(x)})}{R\Gamma(\P^1_{k(x)},T)}\ar[d] &\dfrac{S/E(\A^1_{k(x)})}{R\Gamma(\P^1_{k(x)},S/E)}\ar[r]\ar[l,"\simeq"]\ar[d] &\dfrac{\h0(S)(\A^1_{k(x)})}{R\Gamma(\P^1_{k(x)},\h0(S))}\ar[d]\\
    \dfrac{T(\A^1_{k(x)})}{R\Gamma(\bcube_{k(x)},T)}\ar[d,"\simeq"] &\dfrac{S/E(\A^1_{k(x)})}{R\Gamma(\bcube_{k(x)},S/E)}\ar[d]\ar[r] \ar[l] &\dfrac{\h0(S)(\A^1_{k(x)})}{R\Gamma(\bcube_{k(x)},\h0(S))}\ar[d,"\simeq"]\\
	\dfrac{T(\A^1_{k(x)})}{T(k(x))} &\dfrac{S/E(\A^1_{k(x)})}{S/E(k(x))}\ar[r,"\frac{p_{\A^1_{k(x)}}}{p_{k(x)}}"]\ar[l,"\simeq","\frac{u_{\A^1_{k(x)}}}{u_{k(x)}}"'] &\dfrac{\h0(S)(\A^1_{k(x)})}{\h0(S)(k(x))}.
    \end{tikzcd}
    \]
    where the vertical isomorphisms come from the fact that $h_0^{\bcube}(S)$ and $T$ are strictly $\bcube$-invariant and the horizontal isomorphism from the fact that $\omega_\sharp u$ is an isomorphism.
    
		By taking inverses of $u_{\eta_x}$ and $u_{\A^1_{k(x)}}/u_{k(x)}$ as above, we conclude that the map $\phi$ of \eqref{eq:phi} induces a commutative diagram\[
		\begin{tikzcd}
			T(\eta_x)\ar[d] \ar[rrrr,"\phi_{\eta_x}"]&&&&\h0(S)(\eta_x)\ar[d]\\
			\dfrac{T(\A^1_{k(x)})}{T(k(x))} \ar[rrrr,"\frac{\phi_{\A^1_{k(x)}}}{\phi_{k(x)}}"] &&&&\dfrac{\h0(S)(\A^1_{k(x)})}{\h0(S)(k(x))},
		\end{tikzcd}
		\]
		which implies that $\phi$ respects the residues, so by Proposition \ref{prop:residues}, it lifts uniquely to a map $\phi'\colon T\to \h0(S)$. We need to show that $\phi'\circ f=L$: let $X\in \SmlSm(k)$ and let $X^{\circ}:=\ul{X}-|\partial X|$ be the open with trivial log structure. Since the map $\h0(S)(X)\hookrightarrow \h0(S)(X^{\circ},\triv)$ is injective by purity, Theorem  \ref{thm:purity}, (see \cite[(3.20) and (3.21)]{mericicrys} for a similar argument) it is enough to check the commutativity on $(X^{\circ},\triv)$, which follows automatically by the commutativity of\[
		\begin{tikzcd}
			S/E(X^{\circ})\ar[d,"u_{X^\circ}^{-1}","\simeq"']\ar[r,"p_{X^{\circ}}"] &\h0(S)(X^{\circ})\\
			T(X^{\circ})\ar[ur,"\phi_{X^{\circ}}"']
		\end{tikzcd}\]
	\end{proof}
    \begin{remark}
    The proof works verbatim for the heart of the $t$-structure of $\logDA(k,\Z)$ (see \cite{BindaMerici}).
    \end{remark}

	For $X\in \SmlSm(k)$, we let $X^\circ:=X-|\partial X|\in \Sm(k)$ and $h_i^{\A^1}(X^\circ)$ the $i$-th Suslin homology sheaf. We also write $h_i^{\bcube}(X)$ for $h_i^{\bcube}(\Z_{\ltr}(X))$. We apply the proposition above to prove the following:
	\begin{prop}\label{prop:h0-X}
		Let $X\in \SmlSm(k)$ such that $\ul{X}$ is proper. Then $h_0^{\bcube}(X)\cong \omega^*h_0^{\A^1}(X^\circ)$.
	\end{prop}
	\begin{proof}
		The map $f\colon \Z_{\ltr}(X)\to \omega^*h_0^{\A^1}(X^\circ)$ induced by adjunction from the map\[
		\omega_\sharp \Z_{\ltr}(X) = \Z_{\tr}(X^\circ)\to h_0^{\A^1}(X^\circ)
		\]
		satisfies the fact that $\omega_\sharp f$ is surjective, hence by Proposition \ref{prop;star}, it is enough to show that \eqref{star} holds in this case. Notice that we have the following commutative diagram
        \begin{equation}\label{eq:i0-i1}
		\begin{tikzcd}
			\lCor(\bcube_K,X)\ar[rr,"i_0^*-i_1^*"]\ar[d]&&\Z_{\ltr}(X)(K)\ar[d,equal]\\
			\Cor(\A^1_K,X^\circ)\ar[rr,"{(i_0^{\A^1})^*-(i_1^{\A^1})^*}"]&&\Z_{\tr}(X^\circ)(K)
		\end{tikzcd}
		\end{equation}
		Recall that by \cite[Theorem 13.8]{MVW} \[
  h_0^{\A^1}(X^{\circ})=\pi_0(C_*^{\A^1}(X^{\circ})),\]
  where $C_*^{\A^1}(X^{\circ})$ is the normalized Suslin complex of \cite[Definition 2.14]{MVW}, so we have a right exact sequence \[
  \Cor(\A^1_K,X^{\circ})\xrightarrow{(i_0^{\A^1})^*-(i_1^{\A^1})^*} \Z_{\tr}(X^{\circ})(K)\to h_0^{\A^1}(X^{\circ})(K)\to 0
  \] 
  which implies that\[
  q_K(E(K,\triv))\subseteq {\rm Im}\bigl((i_0^{\A^1})^*-(i_1^{\A^1})^*\bigr)
  \]
		We will conclude by showing that left the vertical map in \eqref{eq:i0-i1} is a bijection: if this holds, then ${\rm Im}\bigl((i_0^{\A^1})^*-(i_1^{\A^1})^*\bigr)={\rm Im}\bigl(i_0^*-i_1^*\bigr)$ and \eqref{star} holds. This map sends a correspondence $Z$ to the correspondence $\omega(Z)$, so it is clearly injective. Let $Y\in \Sm(k)$ with fraction field $K$, so that $\Spec(K)=\lim_{i\in I} Y_i$ with $Y_i\subseteq Y$ open, so that \[
  \lCor(\bcube_K,X) = \colim_{i\in I}\lCor(\bcube_{Y_i},X)\quad\textrm{and}\quad\Cor(\A^1_K,X) = \colim_{i\in I}\Cor(\A^1_{Y_i},X).
  \]
  Let $Z\in \Cor(\A^1_{Y_i}, X^\circ)$. By e.g. \cite[Theorem 1.6.2]{MotModulusI}, there is a proper birational morphism $S\to \P^1_{Y_i}$ which is an isomorphism on $\A^1_{Y_i}$ and a finite correspondence $Z'\colon S\to X$ that restricts to $Z$. By \cite[Corollary 4.4.3(b)]{Liu}, the map $S\to \P^1_{Y_i}$ is an isomorphism outside a closed subscheme $Q\subseteq \P^1_{Y_i}$ of codimension $2$. Let $Q'=\infty_{Y_i}\cap Q$: since $Q$ has codimension $2$, $\infty_{Y_i}-Q'$ is dense in $\infty_{Y_i}$, so there is an open subscheme $Y_j$ of $Y_i$ such that $S\times_{\P^1_{Y_i}}{\P^1_{Y_j}}\to {\P^1_{Y_j}}$ is an isomorphism, which implies that the restriction of $Z'$ induces a finite correspondence $Z''\colon \P^1_{Y_j}\to X$ that restricts to the pullback of $Z$ to $\A^1_{Y_j}$. In particular, the pullback of $|\partial X|$ to $Z''$ is contained in the pullback of $\infty_{Y_j}$, so $Z''\in \lCor(\bcube_{Y_j},X)$ that maps to $Z$ in $\Cor(\A^1_{Y_j},X^{\circ})$, so by passing to the limit this concludes the proof.
	\end{proof}
\begin{remark}
Notice that $\eqref{star}$ is a very strong condition to verify, and indeed Proposition \ref{prop:h0-X} is false if $X=(\ul{X},\triv)$ with $\ul{X}$ affine: indeed in that case the map $\iota$ is not surjective, as $\lCor(\bcube_K,X) = \Cor(\P^1_K,\ul{X}) = \Cor(K,\ul{X})$, therefore $i_0^*-i_1^* = 0$, but $E(K,\triv)$ is in general not zero (e.g. if $X=\A^1$). The correct replacement for this should be coming from the proper modulus pairs as in \cite{MotModulusII}.
In fact, we would like to use the method explained in Proposition \ref{prop;star} to the map considered in \cite{RulSugYama}: \[
		\G_a\otimes^{\ltr} \G_m \to \Omega^1\qquad (a,b)\mapsto a\cdot \dlog(b).
		\]
		We conjecture that this map is an isomorphism, but at the moment we are not able to verify \eqref{star} in this case. We remark that the methods of \cite{RulSugYama} do not apply here since $\Z_{\rm ltr}(\A^1)(\bcube_K)  = \Z_{\rm ltr}(\A^1)(K)$: in order to make it work we need a replacement for the modulus pair $(\P^1,2\infty)$. This is a work in progress.
	\end{remark}
    
	We can now compute the first homotopy groups of $M^{\rm eff}(\P^1,\triv)$ without any assumptions on resolutions of singularities. The square
	\begin{equation}\label{eq:MV-P1}
		\begin{tikzcd}
			(\P^1,0+\infty)\ar[r]\ar[d]&(\P^1,0)\ar[d]\\
			(\P^1,\infty)\ar[r]&(\P^1,\triv),
		\end{tikzcd}
	\end{equation}
	induces a cartesian square in $\logDM(k,\Z)$ (see \cite[Proposition 2.4.9]{BPO-SH}) giving a long exact sequence\[
	\substack{h_1^{\bcube}(\P^1,0)\\ \oplus \\h_1^{\bcube}(\P^1,\infty)}\to h_1^{\bcube}(\P^1,\triv)\to h_0^{\bcube}(\P^1,0+\infty) \to \substack{h_0^{\bcube}(\P^1,0)\\ \oplus \\h_0^{\bcube}(\P^1,\infty)}\to h_0^{\bcube}(\P^1,\triv)\to 0
	\]
	Since the inclusion of the point $[1:1]$ in $\P^1$ induces equivalences in $\logDM(k,\Z)$:\[
	\Z[0]\simeq M^{\rm eff}(\Spec(k))\simeq M^{\rm eff}(\P^1,0) \simeq M^{\rm eff}(\P^1,\infty),
	\] 
	we conclude that $h_0^{\bcube}(\P^1,\triv)\simeq h_0^{\bcube}(\Spec(k))=\Z$. Consider the map $\tau\colon \Z_{\tr}(\P^1)[-1]\to \G_m$ given by the tautological bundle in $\Pic(\P^1)$: by applying $\omega^{\sharp}$ it gives a composition in $\cD(\Sh_{\dNis}^{\tr}(k,\Z))$:
    \begin{equation}\label{eq:comp-omega-bundle}
             \tilde \tau\colon \Z_{\ltr}(\P^1,\triv)[-1]\xrightarrow{\omega^\sharp \tau} \omega^\sharp \G_m \to \omega^*\G_m,
        \end{equation}
    where the second map comes from the adjunction $\omega^\sharp\dashv\omega_\sharp\dashv\omega^*$, which gives natural transformations $\omega^\sharp\to \omega^\sharp\omega_\sharp\omega^*\to \omega^*$. Since $\omega^*\G_m\in \logCI(k,\Z)$, the map factors through a map $M^{\rm eff}(\P^1,\triv)[-1]\to \omega^*\G_m$, and by taking $\pi_0$ we get a map
	\begin{equation}\label{eq:taut-bundle-pi1}
		h_1^{\bcube}(\P^1,\triv)\to \omega^*\G_m.
	\end{equation}
	\begin{prop}\label{prop:h0-gm}
		The map \eqref{eq:taut-bundle-pi1} is an isomorphism.
		\begin{proof} By e.g. \cite[Proposition 5.12]{BindaMerici}, the map above factors through $\omega^*h_1^{\A^1}(\P^1)\to \omega^*\G_m$, which is an isomorphism by \cite[Lecture 4]{MVW}, hence it is enough to prove that $h_1^{\bcube}(\P^1,\triv)\cong \omega^*h_1^{\A^1}(\P^1)$. By the long exact sequences induced by Mayer--Vietoris we have a commutative diagram where the rows are exact sequences:\[
			\begin{tikzcd}
				0\ar[r]&h_1^{\bcube}(\P^1,\triv)\ar[r]\ar[d,"(*1)"] &h_0^{\bcube}(\P^1,0+\infty)\ar[d,"(*2)"]\ar[r]&h_0^{\bcube}(\bcube)^{\times 2}\ar[d,equal]\\
				0\ar[r]&\omega^*h_1^{\A^1}(\P^1)\ar[r] &\omega^*h_0^{\A^1}(\A^1-\{0\})\ar[r]&\omega^*h_0^{\A^1}(\A^1)^{\times 2}.
			\end{tikzcd}\]
			By \cite[Lemma 4.1]{MVW}, we have $h_1^{\A^1}(\A^1-\{0\})=0$, so by Proposition \ref{prop:h0-X} we conclude that the map $(*2)$ above is an isomorphism, so $(*1)$ is an isomorphism too.
		\end{proof}
	\end{prop}
	\begin{remark}
		The computation of weight one motivic cohomology of \cite[Section 4]{MVW}, i.e. that $h_0^{\A^1}(\A^1-\{0\})=\Z\oplus \G_m$ and $h_1^{\A^1}(\A^1-\{0\})=0$, is crucial for applying Proposition \ref{prop:h0-X}. One would be tempted to generalize the result above to the case without transfers: we have by \cite[Theorem 3.37]{Morel} that $\pi_1^{\A^1,{\rm mot}}(\P^1)\cong \sK_1^{\rm MW}$, where the right hand side is the Milnor--Witt $K$-theory sheaf. On the other hand $\Z(\A^1-\{0\})\to \sK_1^{\rm MW}$ is not surjective, so the result does not immediately follow in this case. We leave this (apparently much harder) computation to a future work.
	\end{remark}

    \section{Proof of \texorpdfstring{\cite[Conjecture 0.2]{BindaMericierratum}}{[BM22, Conjecture 0.2]} in the case with transfers}\label{appendix-conjecture}
	
	In this section, we will deduce the aforementioned conjecture from \cite[Theorem 3.26]{mericicrys} (see Theorem \ref{thm:Gysin}). The key point is to show that in this situation the map $\mathrm{Gys}_0$ is already defined in $\Sh_{\Nis}^{\tr}(k,\Z)$. We fix the map $\tau\colon \Z_{\tr}(\P^1)[-1]\to \G_m$ induced by the tautological bundle and the associated map $\tilde \tau\colon \Z_{\ltr}(\P^1,\triv)[-1]\to \omega^*\G_m$.
	\begin{lemma}\label{lemma:trick}
		Let $F\in \logCI(k,\Z)$. The maps $\tau$ and $\tilde \tau$ induce isomorphism\[
		\uHom_{\Sh_{\Nis}^{\tr}}(\G_m,\omega_\sharp F)\cong a_{\Nis}H^1(\P^1\times-,\omega_\sharp F)\cong \omega_\sharp\uHom_{\Sh_{\dNis}^{\ltr}}(\omega^*\G_m,F).
		\]
	\end{lemma}
	\begin{proof}
		Let $X\in \Sm(k)$. Then we have
		\begin{align*}
			\uHom_{\Sh_{\Nis}^{\tr}}(\G_m,\omega_\sharp F)(X) &= \Hom_{\Sh_{\Nis}^{\tr}}(\G_m,\uHom_{\Sh_{\Nis}^{\tr}}(\Z_{\tr}(X),\omega_\sharp F) ) \\
			&=  \Hom_{\Sh_{\Nis}^{\tr}}(\G_m,\omega_{\sharp}\uHom_{\Sh_{\dNis}^{\ltr}}(\Z_{\ltr}(X,\triv),F))
		\end{align*}
		and\[
		\omega_{\sharp}\uHom_{\Sh_{\dNis}^{\ltr}}(\omega^*\G_m,F)(X) = \Hom_{\Sh_{\dNis}^{\ltr}}(\omega^*\G_m,\uHom_{\Sh_{\dNis}^{\ltr}}(\Z_{\ltr}(X,\triv),F)).
		\]
		Since $\uHom_{\Sh_{\dNis}^{\ltr}}(\Z_{\ltr}(X,\triv),F) \in \logCI(k,\Z)$ (see e.g. \cite[Lemma 2.10]{BindaMerici}) and every map $\omega_\sharp F\to \omega_\sharp G$ induces a map
        \begin{align*}
        &\omega_\sharp \uHom_{\Sh_{\dNis}^{\ltr}}(\Z_{\ltr}(X,\triv),F) = \uHom_{\Sh_{\Nis}^{\tr}}(\Z_{\tr}(X),\omega_\sharp F)\\
        &\to \uHom_{\Sh_{\Nis}^{\tr}}(\Z_{\tr}(X),\omega_\sharp G) =  \omega_\sharp \uHom_{\Sh_{\dNis}^{\ltr}}(\Z_{\ltr}(X,\triv),G),    
        \end{align*}
        up to replacing $F$ with $\uHom_{\Sh_{\Nis}^{\tr}}(\Z_{\ltr}(X,\triv),F)$, we are left to prove that the maps
        \begin{equation}\label{eq:1-referee}
		\Hom_{\Sh_{\Nis}^{\tr}}(\G_m,\omega_\sharp F)\to H^1(\P^1,\omega_\sharp F)\leftarrow \Hom_{\Sh_{\dNis}^{\ltr}}(\omega^*\G_m, F)
        \end{equation}
        are isomorphisms. 
        
        By construction the map $\tilde \tau$ is the composition $\Z_{\ltr}(\P^1,\triv)[-1]\xrightarrow{\omega^\sharp \tau} \omega^\sharp \G_m \to \omega^*\G_m$ of \eqref{eq:comp-omega-bundle}.         For all $F\in \Sh^{\ltr}_{\dNis}(k,\Z)$ it induces a commutative diagram functorial in $F$:
        \[
        \begin{tikzcd}
            H^1(\P^1,\omega_\sharp F)&&\Hom_{\Sh_{\Nis}^{\tr}}(\G_m,\omega_\sharp F)\ar[ll,"(*2)"']\\
            &\Hom_{\Sh_{\dNis}^{\ltr}}(\omega^*\G_m,F)\ar[ul,"(*1)"]\ar[ur]
        \end{tikzcd}
        \]
        We first prove that the map $(*1)$ is an isomorphism. Indeed by Proposition \ref{prop:h0-gm} we have that the truncation of the map \eqref{eq:comp-omega-bundle} induces an isomorphism \[
        \psi\colon\Hom_{\Sh_{\dNis}^{\ltr}}(\omega^*\G_m,F)\cong \Hom_{\Sh_{\dNis}^{\ltr}}(h_1^{\bcube}(\P^1,\triv),F).\] 
        Moreover, we have the spectral sequence \[
			E_{2}^{ij}=\Ext^i_{\Sh_{\dNis}^{\ltr}}(h_j^{\bcube}(\P^1,\triv),F)\Rightarrow H^{i+j}(\P^1,\omega_\sharp F)
			\]
            that gives the exact sequence\[
        \begin{aligned}
            &0\to \Ext^1(h_0^{\bcube}(\P^1,\triv),F)\to H^{1}(\P^1,\omega_\sharp F)\xrightarrow{\theta} \Hom(h_1^{\bcube}(\P^1,\triv),F)\to \Ext^2(h_0^{\bcube}(\P^1,\triv),F)
            \end{aligned}
			\]
			We have that $h_0^{\bcube}(\P^1,\triv)=\Z$ and $\Z[0]=M^{\rm eff}(\Spec(k))$, so\[
			\Ext^i(h_0^{\bcube}(\P^1,\triv),F)(X) = H^i(\Spec(k),F) = 0\qquad\textrm{for }i\geq 1,
			\]
			which implies that the map $\theta$ above is an isomorphism given by the truncation, so the composition $\theta^{-1}\circ\psi$ gives $(*1)$.
            
            This implies that in order to show that $(*2)$ above is an isomorphism it is enough to check that it is injective. 
        
        Let $h^0_{\bcube}\omega^*$ be the right adjoint of $\omega_\sharp\iota \colon \logCI(k,\Z)\to \Sh_{\Nis}^{\ltr}(k,\Z)$ (see Section \ref{ssec:logDM} for notation): for all $F\in \logCI(k,\Z)$ the map $F\to h^0_{\bcube}\omega^*\omega_\sharp \iota F$ is injective, as $\iota F\to \omega^*\omega_\sharp \iota F$ is injective, see \eqref{eq:purity-omega}, $h^0_{\bcube}$ is left exact  and $h^0_{\bcube}\iota = \id$ since $\iota$ is fully faithful. Let \[
        F_n:= (h^0_{\bcube}\omega^*\omega_\sharp\iota )^{\circ n} F.\] 
        we have that $F\to F_n$ is injective, so we have a commutative diagram where the vertical maps are injective:
        \[
        \begin{tikzcd}
            H^1(\P^1,\omega_\sharp F)\ar[dd,hook]&&\Hom_{\Sh_{\Nis}^{\tr}}(\G_m,\omega_\sharp F)\ar[ll,near start,"(*2)"']\ar[dd,hook]\\
            &\Hom_{\Sh_{\dNis}^{\ltr}}(\omega^*\G_m,F)\ar[ul,equal]\ar[ur,"(*3)"']\\
            \colim_{n\in \N} H^1(\P^1,\omega_\sharp F_n)&&\colim_{n\in \N} \Hom_{\Sh_{\Nis}^{\tr}}(\G_m,\omega_\sharp F_n).\ar[ll,near start,"(*4)"']\\
            &\colim_{n\in \N} \Hom_{\Sh_{\dNis}^{\ltr}}(\omega^*\G_m,F_n)\ar[ul,equal]\ar[ur,"(*5)"']
            \ar[from = 2-2,hook,crossing over]
        \end{tikzcd}
        \]
        In order to check that $(*2)$ is injective, it is enough to check that the map $(*5)$ above is an isomorphism: indeed, if $(*5)$ is an isomorphism then $(*4)$ is an isomorphism too, so since all the vertical arrows are injective we conclude that $(*2)$ is injective too. On the other hand, we have that the map\[
        \Hom_{\Sh_{\dNis}^{\ltr}}(\omega^*\G_m,F_n)\to \Hom_{\Sh_{\Nis}^{\tr}}(\G_m,\omega_\sharp F_n)
        \]
        equals by adjunction \[
        \Hom_{\Sh_{\dNis}^{\ltr}}(\omega^*\G_m,F_{n}) \to \Hom_{\Sh_{\dNis}^{\ltr}}(\omega^*\G_m,h^0_{\bcube}\omega^*\omega_\sharp\iota  F_{n}) = \Hom_{\Sh_{\Nis}^{\tr}}(\omega^*\G_m,F_{n+1}),
         \]
        therefore the map is just the boundary map $\Hom_{\Sh_{\dNis}^{\ltr}}(\omega^*\G_m,F_{n})\to \Hom_{\Sh_{\dNis}^{\ltr}}(\omega^*\G_m,F_{n+1})$, which by taking colimits on $n$ is an isomorphism. 
        \end{proof} 
        We are now ready to prove our key result. Let $S,T\in \Sh_{\Nis}^{\tr}(k,\Z)$: recall that $\Map(S,T)$ is coconnected, therefore the truncation gives a map \[
        \Hom(S,T)\simeq\tau_{\geq 0} \Map(S,T)\to \Map(S,T),
        \]
        which is functorial for all maps $T\to T'$.
    \begin{lemma}\label{lem:gysin-is-downstairs}
        Let $F\in \logCI(k,\Z)$, $K$ a function field over $k$ and $U\subseteq \A^1_{K}$ a neighborhood of $\{0\}$. After the identification of Lemma \ref{lemma:trick}, the map \[\mathrm{Gys}_0^F(U)\colon H^1(\P^1_{K},\omega_\sharp F)\to \omega_\sharp F(U)[1]\] agrees with the composition \[
        \Hom(\G_{m}\otimes\Z_{\tr}(K),\omega_\sharp F)\to \Map(\G_{m}\otimes\Z_{\tr}(K),\omega_\sharp F)\xrightarrow{\alpha} \Map(\Z_{\tr}(\P^1_{K}),\omega_\sharp F)[1]\xrightarrow{\beta} \omega_\sharp F(U)[1]
        \]
        where $\alpha$ is the map induced by $\tau\colon \Z_{\tr}(\P^1_K)[-1]\to \G_m\otimes \Z_{\tr}(K)$ 
        and $\beta$ is the map induced by $\Z_{\tr}(U)[-1]\to \Z_{\tr}(\P^1_{K})[-1]$. 
    \end{lemma}
    \begin{proof}
    \def\uMap{\underline{\Map}}
 By the choice of a trivialization of the tautological bundle on $\P^1-\{0\}$, the map $\tau$ factors through \[
    \delta\colon \Z_{\tr}(\P^1)/\Z_{\tr}(\P^1-\{0\})[-1]\to \G_m,\] 
and we have a commutative diagram \[
\begin{tikzcd}
\omega^\sharp \Z_{\tr}(\P^1)\ar[r]\ar[d,"\cong"]&
\dfrac{\omega^\sharp \Z_{\tr}(\P^1)}{\omega^\sharp \Z_{\tr}(\P^1-\{0\})}\ar[d]\\
\Z_{\ltr}(\P^1,\triv)\ar[r]&\dfrac{\Z_{\ltr}(\P^1,\triv)}{\Z_{\ltr}(\P^1,0)}
\end{tikzcd}
\]
    Then the truncation maps together with the isomorphisms of Lemma \ref{lemma:trick} fit in the following commutative diagram.
    \begin{equation}\label{eq:lemma-trick-last}
    \begin{tikzcd}
         a_{\Nis}H^1(\P^1\times-,\omega_\sharp F)\ar[d]\ar[r,"\cong"]&\uHom(\G_m,\omega_\sharp F)\ar[r]&\uMap(\G_m,\omega_\sharp F)\ar[dl,"F(\delta)"]\ar[d,"F(\tau)"]\\
      \omega_\sharp\uMap(\dfrac{\Z_{\ltr}(\P^1,\triv)}{\Z_{\ltr}(\P^1,0)}[-1],F) \ar[r]&\uMap(\dfrac{\Z_{\tr}(\P^1)}{\Z_{\tr}(\P^1-\{0\})}[-1],\omega_\sharp F)\ar[r]&\uMap({\Z_{\tr}(\P^1)}[-1],\omega_\sharp F)
    \end{tikzcd}
    \end{equation}
    
    We have the commutative diagram obtained by applying $\omega^\sharp\omega_\sharp$ to \eqref{eq:gysin-pushout} :
            \[
            \begin{tikzcd}
            &&\omega^\sharp\Z_{\tr}(\P^1_{K})\ar[dr,"(*3)"]\ar[ddd,near start,"\epsilon_{\P^1_K}"]\\
                \omega^{\sharp }\Z_{\ltr}(U)\ar[dr,"(*2)"']\ar[urr,near start,"(*1)"]\ar[ddd,near start,"\epsilon_{U}"']&&&\dfrac{\omega^\sharp\Z_{\tr}(\P^1_{K})}{\omega^\sharp\Z_{\tr}(\P^1_{K}-\{0\})}\ar[ddd,near start, "\epsilon_{(\P^1_K,0)}"]\\
                &\dfrac{\omega^\sharp\Z_{\tr}(U)}{\omega^\sharp\Z_{\tr}(U-\{0\})}\ar[urr, crossing over,near end,"(*4)"']\\
                &&\Z_{\ltr}(\P^1_{K},\triv)\ar[dr,"(*7)"]\\
                \Z_{\ltr}(U,\triv)\ar[dr,"(*6)"']\ar[urr,near start,"(*5)"] &&& \dfrac{\Z_{\ltr}(\P^1_{K},\triv)}{\Z_{\ltr}(\P^1_{K},0)}\\
                &
                \dfrac{\Z_{\ltr}(U,\triv)}{\Z_{\ltr}(U,0)}\ar[urr,near end,"(*8)"']\ar[from=3-2, crossing over,near start,"\epsilon_{(U,0)}"],
            \end{tikzcd}            
            \]
By applying the functor $\Map((-)[-1],F)$ to the diagram above (and using the adjunction $\omega^\sharp\dashv \omega_\sharp$) and combining it with \eqref{eq:lemma-trick-last} evaluated at $K$, we have the following commutative diagram:
     \[\begin{tikzcd}
         H^1(\P^1_K,\omega_\sharp F)\ar[dd,"\cong"]\ar[r,"\textrm{Lemma \ref{lemma:trick}}","\cong"']&\Hom(\G_m\otimes\Z_{\tr}(K),\omega_\sharp F)\ar[d]\ar[dddr,bend left=15, "\alpha"]\\
         &\Map(\G_m\otimes\Z_{\tr}(K),\omega_\sharp F)\ar[d,"F(\delta)"]\\
      \Map\Bigl(\dfrac{\Z_{\ltr}(\P^1_K,\triv)}{\Z_{\ltr}(\P^1_K,0)}[-1],F\Bigr) \ar[r,"F(\epsilon_{(\P^1,0)})"]\ar[d,"F(*8)"]&\Map\Bigl(\dfrac{\Z_{\tr}(\P^1_K)}{\Z_{\tr}(\P^1_K-\{0\})}[-1],\omega_\sharp F\Bigr)\ar[d,"F(*4)"]\ar[dr,bend left=10,"F(*3)"']\\
      \Map\Bigl(\dfrac{\Z_{\ltr}(U,\triv)}{\Z_{\ltr}(U,0)}[-1],F\Bigr)\ar[r,"F(\epsilon_{(U,0)})"]\ar[d,"F(*6)"]&\Map\Bigl(\dfrac{\Z_{\tr}(U)}{\Z_{\tr}(U-\{0\})}[-1],\omega_\sharp F\Bigr)\ar[d,"F(*2)"]&\Map({\Z_{\tr}(\P^1_K)}[-1],\omega_\sharp F)\ar[dl,bend left=10,"F(*1)"']\ar[ddl,bend left = 10,"\beta"]\\
      \Map({\Z_{\ltr}(U,\triv)}[-1],F)\ar[dr,"\simeq"]\ar[r,"F(\epsilon_{U})"]&\Map({\Z_{\tr}(U)}[-1],\omega_\sharp F)\ar[d,"\simeq"]\\
      &\omega_\sharp F(U)[1]
    \end{tikzcd}
     \]
     Then by construction (see \eqref{eq:Thom}) the composition through the left vertical map is $\mathrm{Gys}_0^F(U)$, so we conclude.
    \end{proof}
 \begin{remark}\label{rmk:funct}
 Let $F,G\in \logCI(k,\Z)$ and let $\phi\colon \omega_\sharp F\to \omega_\sharp G$. We have that
 \begin{enumerate}
 \item The map in Lemma \ref{lemma:trick} fits into a commutative diagram of abelian groups\[
	\begin{tikzcd} 
 H^1(\P^1_K,\omega_\sharp F)\ar[rr,"\rm Lemma\ \ref{lemma:trick}","\cong"']\ar[d,"{H^1(\P^1_K,\phi)}"]&&\Hom(\G_m\otimes\Z_{\tr}(K),\omega_\sharp F)\ar[d,"{\Hom(\G_m\otimes\Z_{\tr}(K),\phi)}"]\\
 H^1(\P^1_K,\omega_\sharp G)\ar[rr,"\rm Lemma\ \ref{lemma:trick}","\cong"']&&\Hom(\G_m\otimes\Z_{\tr}(K),\omega_\sharp G)
\end{tikzcd} 
 \]
 as it is the inverse of the map $\Hom(\G_m\otimes\Z_{\tr}(K),-)\to  H^1(\P^1_K,-)$ induced by $\tau$.
 \item The maps involved in the composition of $\alpha$ and $\beta$ are functorial for $\omega_\sharp F\to \omega_\sharp G$ by construction, therefore there is a commutative diagram in $\cD(\mathbf{Ab})$:\[
 \begin{tikzcd}[column sep=small]
 \Hom(\G_m\otimes\Z_{\tr}(K),\omega_\sharp F)\ar[d,"{\Hom(\G_m\otimes\Z_{\tr}(K),\phi)}"]\ar[r]&\Map(\G_m\otimes\Z_{\tr}(K),\omega_\sharp F)\ar[r,"\alpha"]\ar[d,"{\Map(\G_m\otimes\Z_{\tr}(K),\phi)}"]&\Map(\Z_{\tr}(\P^1)[-1],\omega_\sharp F)\ar[r,"\beta"]\ar[d,"{\Map(\Z_{\tr}(\P^1)[-1],\phi)}"]&\omega_\sharp F(U)[1]\ar[d,"\phi_{\A^1}(U)"]\\
 \Hom(\G_m\otimes\Z_{\tr}(K),\omega_\sharp F)\ar[r]\ar[r]&\Map(\G_m\otimes\Z_{\tr}(K),\omega_\sharp F)\ar[r,"\alpha"]&\Map(\Z_{\tr}(\P^1)[-1],\omega_\sharp G)\ar[r,"\beta"]&\omega_\sharp F(U)[1]
 \end{tikzcd}
 \]
 \item All the maps above are functorial for open immersions $0\in V\subseteq U\subseteq \P^1_K$ by construction.
  \end{enumerate}
 \end{remark}
	We are ready to settle \cite[Conjecture 0.2]{BindaMericierratum} in the case with transfers:
	\begin{thm}\label{thm:almost-conj}
		Let $F,G\in \logCI(k,\Z)$, and let $\phi\colon \omega_\sharp F\to \omega_\sharp G$ be a map in $\Sh_{\Nis}^{\tr}$. Then $\phi$ lifts uniquely to a map $F\to G$ in $\logCI(k,\Z)$. In particular the functor $
        \logCI(k,\Z)\xrightarrow{\omega_\sharp\iota} \Sh_{\Nis}^{\tr}$
        is fully faithful and exact and, for all $F\in \logCI(k,\Z)$, we have $h_0^{\bcube}\omega^\sharp\omega_\sharp\iota F\cong F$.
		\begin{proof}
            By Lemma \ref{lem:gysin-is-downstairs} and Remark \ref{rmk:funct}, the maps $\mathrm{Gys}_0^F$ and $\mathrm{Gys}_0^G$ fit in the commutative diagram of Theorem \ref{thm:Gysin}, therefore there exists a unique $\tilde{\phi}$ lifting $\phi$, therefore the (exact) functor $\logCI(k,\Z)\xrightarrow{\omega_\sharp\iota}\Sh_{\Nis}^{\tr}(k)$ is fully faithful. Finally, $\omega_\sharp\iota$ has a left adjoint given by $h_0^{\bcube}\omega^\sharp$, so the last part follows from fully faithfulness. 
		\end{proof}
	\end{thm}
		\bibliographystyle{alpha}
	\bibliography{bibMerici}
\end{document}